\documentclass[12pt]{article}%
\usepackage{amsmath}
\usepackage{amsfonts}
\usepackage{amssymb}
\usepackage{graphicx}
\providecommand{\U}[1]{\protect\rule{.1in}{.1in}}
\newtheorem{theorem}{Theorem}

\newtheorem{example}[theorem]{Example}

\newtheorem{lemma}[theorem]{Lemma}

\newtheorem{remark}[theorem]{Remark}

\newenvironment{proof}[1][Proof]{\noindent\textbf{#1.} }{\ \rule{0.5em}{0.5em}}
\begin{document}

\title{Unitary Representations of Lattices of Free Nilpotent Lie Groups of Step-Two }
\author{Vignon Oussa\\Dept.\ of Mathematics \\Bridgewater State University\\Bridgewater, MA 02325 U.S.A.\\}
\date{June 2013}
\maketitle

\begin{abstract}
Using a theorem proved by Bekka and Driutti, we show that if $\mathfrak{f}$ is a
freely generated nilpotent Lie algebra of step-two, then almost every
irreducible representation of the corresponding Lie group restricted to some
lattice $\Gamma$ is an irreducible representation of $\Gamma$ if the dimension of the Lie algebra is odd. However, if the dimension of
the Lie algebra is even, then almost every unitary irreducible representation
of the Lie group restricted to $\Gamma$ is reducible.

\end{abstract}

\section{Introduction}
As Gabor theory continues to surge in popularity, it is increasingly drawing a lot of
attention to the representation theory of nilpotent groups. This
unusual connection is easily discovered through Fourier analysis. In fact, the
conjugation of a translation operator by the Plancherel transform defined on
$L^{2}(\mathbb{R})$ is a modulation operator. The group generated by the
continuous family of translation and modulation operators is called the
reduced Heisenberg group. Its universal covering group is a simply connected,
connected nilpotent Lie group with Lie algebra spanned by $X_{1},X_{2},X_{3}$
with non trivial Lie brackets $[X_{3},X_{2}]=X_{1}.$ In fact the Heisenberg
Lie algebra is a free nilpotent Lie algebra of step-two and two generators.
The set of infinite dimensional unitary irreducible representations of the Heisenberg
group  up to equivalence is parametrized by the punctured line $%
%TCIMACRO{\U{211d} }%
%BeginExpansion
\mathbb{R}
%EndExpansion
^{\ast}$ (Chapter $7$, \cite{Folland}). For each element in the punctured
line, the corresponding unitary irreducible representation is a
Schr\"{o}dinger representation which plays a central role in Gabor theory. Let
$\pi_{\lambda}$ be a Schr\"{o}dinger representation. This representation acts
on $L^{2}\left(
%TCIMACRO{\U{211d} }%
%BeginExpansion
\mathbb{R}
%EndExpansion
\right)  $ as follows
\begin{align*}
\pi_{\lambda}\left(  \exp x_{3}X_{3}\right)  F\left(  t\right)   &  =F\left(
t-x_{3}\right)  \\
\pi_{\lambda}\left(  \exp x_{2}X_{2}\right)  F\left(  t\right)   &  =e^{-2\pi
i\lambda tx_{2}}F\left(  t\right)  \\
\pi_{\lambda}\left(  \exp x_{1}X_{1}\right)  F\left(  t\right)   &  =e^{2\pi
i\lambda x_{1}}F\left(  t\right)  .
\end{align*}
Thus, the family of functions $\pi_{\lambda}\left(  \exp\left(
%TCIMACRO{\U{2124} }%
%BeginExpansion
\mathbb{Z}
%EndExpansion
X_{2}\right)  \exp\left(
%TCIMACRO{\U{2124} }%
%BeginExpansion
\mathbb{Z}
%EndExpansion
X_{3}\right)  \right)  F$ is a Gabor system (see \cite{Han}). Now, if
$\left\vert \lambda\right\vert \leq1,$ then there exists a function $F$ such that
the countable family of vectors $\pi_{\lambda}\left(  \exp\left(
%TCIMACRO{\U{2124} }%
%BeginExpansion
\mathbb{Z}
%EndExpansion
X_{2}\right)  \exp\left(
%TCIMACRO{\U{2124} }%
%BeginExpansion
\mathbb{Z}
%EndExpansion
X_{3}\right)  \right)  F$ is a complete set in $L^{2}\left(
%TCIMACRO{\U{211d} }%
%BeginExpansion
\mathbb{R}
%EndExpansion
\right)  $ and forms what is called a Gabor wavelet system (see \cite{Han}).
In fact, the construction of Gabor wavelets is a very active area of research
(see \cite{Han} \cite{Heil}). We remark that although the restriction of
$\pi_{\lambda}$ to the discrete group $\exp\left(
%TCIMACRO{\U{2124} }%
%BeginExpansion
\mathbb{Z}
%EndExpansion
X_{1}\right)  \exp\left(
%TCIMACRO{\U{2124} }%
%BeginExpansion
\mathbb{Z}
%EndExpansion
X_{2}\right)  \exp\left(
%TCIMACRO{\U{2124} }%
%BeginExpansion
\mathbb{Z}
%EndExpansion
X_{2}\right)  $ is a cyclic representation if $\left\vert \lambda\right\vert
\leq1,$ it is clearly highly reducible. In contrast to the Heisenberg Lie
algebra, let us consider the free nilpotent Lie algebra of step-two with three
generators: $X_{1},X_{2},X_{3},$ whose Lie algebra is spanned by $X_{1}%
,X_{2},X_{3},Z_{1},Z_{2},Z_{3}$ such that the only non-trivial Lie brackets
are $
\left[  X_{1},X_{2}\right]  =Z_{1},\left[  X_{1},X_{3}\right]  =Z_{2},\left[
X_{2},X_{3}\right]  =Z_{3}.$ Let $\mu$ be the Lebesgue measure on the dual of $%
%TCIMACRO{\U{211d} }%
%BeginExpansion
\mathbb{R}
%EndExpansion
$-span $\left\{  X_{1},X_{2},X_{3},Z_{1},Z_{2},Z_{3}\right\}  .$ Let us
consider an arbitrary infinite dimensional irreducible representation
$\pi_{\lambda}$ corresponding to a linear functional
$\lambda$ via Kirillov's map (see \cite{Corwin}.) The restriction of
$\pi_{\lambda}$ to the discrete group
\[
\Gamma=\exp\left(  \mathbb{Z}Z_{1}\right)  \exp\left(   \mathbb{Z}Z_{2}\right)
\exp\left(  \mathbb{Z}Z_{3}\right)  \exp\left(  \mathbb{Z}X_{1}\right)  \exp\left(
 \mathbb{Z}X_{2}\right)  \exp\left(  \mathbb{Z}X_{3}\right)
\]
is also irreducible for $\mu$-a.e. $\lambda$ in the dual of the Lie algebra. This surprising fact is easily explained. In fact, using the orbit method (see \cite{Corwin}), and some formal calculations, it is not hard to see that almost every infinite dimensional irreducible representation of the group is parametrized by the manifold
\[
\Lambda=\left\{  \left(  \lambda_{1},\lambda_{2},\lambda_{3},\lambda
_{4}\right)  \in%
%TCIMACRO{\U{211d} }%
%BeginExpansion
\mathbb{R}
%EndExpansion
^{4}:\lambda_{3}\neq0\right\}.
\]
For each $\left(  \lambda_{1},\lambda_{2},\lambda_{3},\lambda_{4}\right)
\in\Lambda,$ the corresponding irreducible representation is realized as acting in $L^2(\mathbb{R})$ as follows%
\begin{align*}
\pi_{\left(  \lambda_{1},\lambda_{2},\lambda_{3},\lambda_{4}\right)  }\left(
\exp x_{3}X_{3}\right)  F\left(  t\right)    & =e^{-2\pi itx_{3}\lambda_{3}%
}F\left(  t\right)  \\
\pi_{\left(  \lambda_{1},\lambda_{2},\lambda_{3},\lambda_{4}\right)  }\left(
\exp x_{2}X_{2}\right)  F\left(  t\right)    & =F\left(  t-x_{2}\right)  \\
\pi_{\left(  \lambda_{1},\lambda_{2},\lambda_{3},\lambda_{4}\right)  }\left(
\exp x_{1}X_{1}\right)  F\left(  t\right)    & =e^{2\pi ix_{1}\lambda_{4}%
}e^{2\pi itx_{1}\lambda_{1}}e^{-2\pi i\frac{x_{1}^{2}\lambda_{2}\lambda_{1}%
}{\lambda_{3}}}F\left(  t-\frac{\lambda_{2}}{\lambda_{3}}x_{1}\right)  \\
\pi_{\left(  \lambda_{1},\lambda_{2},\lambda_{3},\lambda_{4}\right)  }\left(
\exp z_{k}Z_{k}\right)  F\left(  t\right)    & =e^{2\pi iz_{k}\lambda_{k}%
}F\left(  t\right)  .
\end{align*}
\noindent
Based on the action of the irreducible representation described above, it is clear that in fact $\pi_{\lambda}|_{\Gamma}$ is irreducible a.e. Before we introduce the general case considered in this paper, we recall a theorem which is due to Bekka and Driutti \cite{Bekka}.

\begin{theorem}
\label{Bekka}Let $N$ be a nilpotent Lie group with rational structure. Let
$\Gamma$ be a lattice subgroup of $N.$ Let $\lambda\in\mathfrak{n}^{\ast}$ and
$\pi_{\lambda}$ its corresponding representation. Then the representation
$\pi_{\lambda}|_{\Gamma}$ which is the restriction of $\pi_{\lambda}$ to
$\Gamma$ is irreducible if and only if the null-space of the matrix $\left[
\lambda\left[  X_{i},X_{j}\right]  \right]  _{1\leq i,j,n}$ with respect to a
fixed Jordan H\"{o}lder basis of the Lie algebra is not contained in a proper
rational ideal of $\mathfrak{n.}$
\end{theorem}

Bekka and Driutti have also provided a rather simple algorithm to determine
whether a sub-algebra of $\mathfrak{n}$ is contained in a rational ideal of
$\mathfrak{n}$ or not (Proposition $1.1.$ \cite{Bekka}).

\begin{enumerate}
\item Let $\mathfrak{h}$ be a sub-algebra of $\mathfrak{n.}$ Fix a Jordan
H\"{o}lder basis $\{X_1,\cdots X_n\}$  for $\mathfrak{n}$ passing through $\left[  \mathfrak{n,n}%
\right]  $

\item If there exists $X\in\mathfrak{h}$ such that $
X=%
%TCIMACRO{\dsum \limits_{i=1}^{n}}%
%BeginExpansion
{\displaystyle\sum\limits_{i=1}^{n}}
%EndExpansion
x_{i}X_{i}$ and if
\[
\dim_{%
%TCIMACRO{\U{211a} }%
%BeginExpansion
\mathbb{Q}
%EndExpansion
}\left(  x_{\dim\left[  \mathfrak{n,n}\right]  +1},\cdots,x_{n}\right)
=n-\dim\left[  \mathfrak{n,n}\right]
\]
then $\mathfrak{h}$ is \textbf{not} contained in a proper rational ideal of
$\mathfrak{n}$

\item Otherwise, $\mathfrak{h}$ is contained in a proper rational ideal of
$\mathfrak{n}$
\end{enumerate}

Coming back to the example of the Heisenberg group which was discussed
earlier, we observe that for each linear functional $\lambda\in\mathfrak{%
%TCIMACRO{\U{211d} }%
%BeginExpansion
\mathbb{R}
%EndExpansion
}^{\ast}$ satisfying $\lambda\left(  X_{1}\right)  \neq0,$ the nullspace of
the matrix
\[
\left[
\begin{array}
[c]{ccc}%
0 & 0 & 0\\
0 & 0 & \lambda\left(  X_{1}\right) \\
0 & -\lambda\left(  X_{1}\right)  & 0
\end{array}
\right]
\]
with respect to the ordered basis $\left\{  X_{1},X_{2},X_{3}\right\}  $ is
the center of the Lie algebra which is a rational ideal itself. Thus according
to Theorem \cite{Bekka}, the restriction of every Schr\"{o}dinger
representation to the lattice $\exp\left(
%TCIMACRO{\U{2124} }%
%BeginExpansion
\mathbb{Z}
%EndExpansion
X_{1}\right)  \exp\left(
%TCIMACRO{\U{2124} }%
%BeginExpansion
\mathbb{Z}
%EndExpansion
X_{2}\right)  \exp\left(
%TCIMACRO{\U{2124} }%
%BeginExpansion
\mathbb{Z}
%EndExpansion
X_{2}\right)  $ is always reducible.

Using the theorem and algorithm above, we are able to show the following. Let
$\mathfrak{f}_{m,2}$ be the free nilpotent Lie algebra of step-two on $m$
generators $\left(  m>1\right)  .$ Let $\Gamma$ be a lattice subgroup of
$\mathfrak{f}_{m,2}$ and let $\mu$ be the canonical Lebesgue measure on
$\mathfrak{f}_{m,2}^{\ast}.$ Then

\begin{theorem}
\label{main2} If $m$ is \textbf{odd} then for $\lambda\in\mathfrak{n}^{\ast}$,
$\pi_{\lambda}|_{\Gamma}$ is irreducible $\mu$-a.e.
\end{theorem}

\begin{theorem}
\label{main1} If $m$ is \textbf{even} then for $\lambda\in\mathfrak{n}^{\ast}%
$, $\pi_{\lambda}|_{\Gamma}$ is reducible $\mu$-a.e.
\end{theorem}

This paper is organized as follows. In the second section of the paper, we
introduce some basic facts about free nilpotent Lie algebras of step-two. In
the third section of the paper, we prove Theorem \ref{main2}, and in the last
section, we prove Theorem \ref{main1}.

\subsection{Preliminaries}

Let $\mathfrak{n}$ be a nilpotent Lie algebra of dimension $n$ over
$\mathbb{R}$ with corresponding Lie group $N=\exp\mathfrak{n}$. We assume that
$N$ is simply connected and connected. It is well-known that if we write
\begin{align*}
\mathfrak{n}_{1}\text{ } &  \mathfrak{=}\text{ }\left[  \mathfrak{n,n}\right]
\\
&  \vdots\\
\mathfrak{n}_{k}\text{ } &  \mathfrak{=}\text{ }\left[  \mathfrak{n,n}%
_{k-1}\right]  \\
&  \vdots
\end{align*}
then there exists $s\in%
%TCIMACRO{\U{2115} }%
%BeginExpansion
\mathbb{N}
%EndExpansion
$ such that $\mathfrak{n}_{s}=\left\{  0\right\}  .$ Let $\mathfrak{v}$ be a
subset of $\mathfrak{n}$ and let $\lambda$ be a linear functional in
$\mathfrak{n}^{\ast}$. We define the corresponding sets $\mathfrak{v}%
^{\lambda}$ and $\mathfrak{v}\left(  \lambda\right)  $ such that
\[
\mathfrak{v}^{\lambda}=\left\{  Z\in\mathfrak{n:}\text{ }\lambda\left[
Z,X\right]  =0\text{ for every }X\in\mathfrak{v}\right\}
\]
and $\mathfrak{v}\left(  \lambda\right)  =\mathfrak{v}^{\lambda}%
\cap\mathfrak{v.}$ $\mathfrak{z}$ denotes the center of the Lie algebra of
$\mathfrak{n,}$ and the coadjoint action on the dual of $\mathfrak{n}$ is
simply the dual of the adjoint action of $N$ on $\mathfrak{n}$. Given
$X\in\mathfrak{n},\lambda\in\mathfrak{n}^{\ast}$, the coadjoint action is
defined multiplicatively as follows:
\[
\exp X\cdot\lambda\left(  Y\right)  =\lambda\left(  Ad_{\exp-X}Y\right)  .
\]
We fix a Jordan H\"{o}lder basis $\left\{  X_{i}\right\}  _{i=1}^{n}$ for
$\mathfrak{n}$. Given any linear functional $\lambda\in\mathfrak{n}^{\ast},$
we construct the following skew-symmetric matrix:
\begin{equation}
\mathbf{M}\left(  \lambda\right)  =\left[  \lambda\left[  X_{i},X_{j}\right]
\right]  _{1\leq i,j,n}.\label{Ml}%
\end{equation}
It is easy to see that $
\mathfrak{n}\left(  \lambda\right)  =\mathrm{nullspace}\left(  \mathbf{M}%
\left(  \lambda\right)  \right)$ with respect to the fixed Jordan H\"{o}lder basis, and that the
center of the Lie algebra is always contained inside the vector space
$\mathfrak{n}\left(  \lambda\right)  .$ It is also well-known that all
coadjoint orbits have a natural symplectic smooth structure, and therefore are
even-dimensional manifolds. Also, thanks to Kirillov's theory, it is
well-known that for each $\lambda\in$ $\mathfrak{n}^{\ast},$ there is a
corresponding analytic subgroup $P_{\lambda}=\exp\left(  \mathfrak{p}%
(\lambda)\right)  $ such that $\mathfrak{p}(\lambda)$ is a maximal sub-algebra
of $\mathfrak{n}$ which is self orthogonal with respect to the bilinear form
\[
\left(  X,Y\right)  \mapsto\lambda\left[  X,Y\right]  .
\]
There is also a character $\chi_{\lambda}$ of $P_{\lambda}$ such that the pair
$\left(  P_{\lambda},\chi_{\lambda}\right)  $ determines up to unitary
equivalence a unique irreducible representation of $N.$ More precisely, if
$\chi_{\lambda}\left(  \exp X\right)  =e^{2\pi i\lambda\left(  X\right)  }$
defines a character on $P_{\lambda}$ then the unitary representation of $N$
\[
\pi_{\lambda}=\mathrm{Ind}_{P_{\lambda}}^{N}\left(  \chi_{\lambda}\right)
\]
is irreducible. This construction exhausts the set of all unitary irreducible
representations of $N.$ We refer the interested reader to \cite{Corwin} which
is a standard reference for this class of Lie groups. Now, we will recall some
basic facts about nilpotent Lie groups which are also found in \cite{Corwin}

\begin{enumerate}
\item $\mathfrak{n}$ has a \textbf{rational structure} if and only if there is
an $%
%TCIMACRO{\U{211d} }%
%BeginExpansion
\mathbb{R}
%EndExpansion
$-basis $\mathfrak{B}$ for $\mathfrak{n}$ having rational structure constants
and if $\mathfrak{n}_{\mathfrak{%
%TCIMACRO{\U{211a} }%
%BeginExpansion
\mathbb{Q}
%EndExpansion
}}=\mathfrak{%
%TCIMACRO{\U{211a} }%
%BeginExpansion
\mathbb{Q}
%EndExpansion
}$ -span $\left(  \mathfrak{B}\right)  $ then
\[
\mathfrak{n\cong\mathfrak{n}_{\mathfrak{%
%TCIMACRO{\U{211a} }%
%BeginExpansion
\mathbb{Q}
%EndExpansion
}}\otimes%
%TCIMACRO{\U{211a} }%
%BeginExpansion
\mathbb{Q}
%EndExpansion
.}%
\]

\item If $N$ has a uniform subgroup $\Gamma$ then $\mathfrak{n}$ has a
rational structure such that $\mathfrak{\mathfrak{n}_{\mathfrak{%
%TCIMACRO{\U{211a} }%
%BeginExpansion
\mathbb{Q}
%EndExpansion
}}=%
%TCIMACRO{\U{211a} }%
%BeginExpansion
\mathbb{Q}
%EndExpansion
}$ -span $\left(  \log\left(  \Gamma\right)  \right)  $

\item If $\mathfrak{n}$ has a rational structure then $N$ has a uniform
subgroup $\Gamma$ such that $\log\left(  \Gamma\right)  \subseteq
\mathfrak{n}_{\mathfrak{%
%TCIMACRO{\U{211a} }%
%BeginExpansion
\mathbb{Q}
%EndExpansion
}}$
\end{enumerate}

\section{Free Nilpotent Lie Algebras of Step Two and Examples}

Let $\mathfrak{f}_{m,2}$ be the free nilpotent Lie algebra of step two on $m$
generators $\left(  m>1\right)  $. Let $Z_{1},\cdots,Z_{m}$ be the generators
of $\mathfrak{f}_{m,2}.$ Then
\[
\mathfrak{f}_{m,2}=\mathfrak{z}\oplus%
%TCIMACRO{\U{211d} }%
%BeginExpansion
\mathbb{R}
%EndExpansion
\text{-span }\left\{  Z_{1},\cdots,Z_{m}\right\}
\]
such that $
\mathfrak{z}=\text{ }%
%TCIMACRO{\U{211d} }%
%BeginExpansion
\mathbb{R}
%EndExpansion
\text{-span }\left\{  Z_{ik}:1\leq i\leq m\text{, }i<k\leq m\right\} .$ The Lie brackets of this Lie algebra are described as follows.
\[
\left[  Z_{i},Z_{j}\right]  =Z_{ij\text{ }}\text{for }\left(  1\leq i\leq
m\text{ and }i<j\leq m\right)
\]
It is then easy too see that $\dim\left(  \mathfrak{z}\right)  \mathfrak{=}$ $\frac{m\left(  m-1\right)  }{2}.$ Also, it is
convenient to relabel the basis elements of the Lie algebra as follows:
\[
\left\{
\begin{array}
[c]{cc}%
X_{1}=Z_{12} & X_{\frac{m\left(  m-1\right)  }{2}+1}=Z_{1}\\
X_{2}=Z_{13} & \vdots\\
\vdots & X_{n-1}=Z_{m-1}\\
X_{\frac{m\left(  m-1\right)  }{2}}=Z_{m-1m} & X_{n}=Z_{m}%
\end{array}
\right\}
\]
Clearly, $\left\{  X_{1},X_{2},\cdots,X_{n-1},X_{n}\right\}  $ is a Jordan
H\"{o}lder basis through $\mathfrak{z}=\left[
\mathfrak{f}_{m,2},\mathfrak{f}_{m,2}\right]  $ which is fixed from now on. We
remark that since $\mathfrak{f}_{m,2}$ has a rational structure, then its
admits a lattice subgroup which we will denote throughout this paper by
$\Gamma.$

\begin{example}
Let us suppose that $m=5.$ We write $\lambda_{ij}=\lambda\left[  Z_{i}%
,Z_{j}\right]  $ for $i<j$ and $\lambda\in\mathfrak{f}_{5,2}^{\ast}$. Put
\[
\Omega=\left\{  f\in\mathfrak{f}_{5,2}^{\ast}:f_{14}f_{23}-f_{13}f_{24}%
+f_{12}f_{34}\neq0\right\}  .
\]
With some formal calculations, we obtain for $\lambda\in\Omega,$ that
\[
\mathfrak{f}_{m,2}\left(  \lambda\right)  =\mathfrak{z\oplus%
%TCIMACRO{\U{211d} }%
%BeginExpansion
\mathbb{R}
%EndExpansion
}\left(  \alpha_{1}\left(  \lambda\right)  Z_{1}+\alpha_{2}\left(
\lambda\right)  Z_{2}+\alpha_{3}\left(  \lambda\right)  Z_{3}+\alpha
_{4}\left(  \lambda\right)  Z_{4}+Z_{5}\right)
\]
where
\[
\alpha_{k}\left(  \lambda\right)  =\left\{
\begin{array}
[c]{cc}%
\dfrac{\lambda_{25}\lambda_{34}-\lambda_{24}\lambda_{35}+\lambda_{23}%
\lambda_{45}}{\lambda_{14}\lambda_{23}-\lambda_{13}\lambda_{24}+\lambda
_{12}\lambda_{34}} & \text{ if }k=1\\
\dfrac{-\lambda_{15}\lambda_{34}-\lambda_{14}\lambda_{35}+\lambda_{13}%
\lambda_{45}}{\lambda_{14}\lambda_{23}-\lambda_{13}\lambda_{24}+\lambda
_{12}\lambda_{34}} & \text{if }k=2\\
\dfrac{\lambda_{15}\lambda_{24}-\lambda_{14}\lambda_{25}+\lambda_{12}%
\lambda_{45}}{\lambda_{14}\lambda_{23}-\lambda_{13}\lambda_{24}+\lambda
_{12}\lambda_{34}} & \text{if }k=3\\
\dfrac{-\lambda_{15}\lambda_{23}+\lambda_{13}\lambda_{25}+\lambda_{12}%
\lambda_{35}}{\lambda_{14}\lambda_{23}-\lambda_{13}\lambda_{24}+\lambda
_{12}\lambda_{34}} & \text{if }k=4
\end{array}
\right.  .
\]
Clearly $
\dim_{%
%TCIMACRO{\U{211a} }%
%BeginExpansion
\mathbb{Q}
%EndExpansion
}\left(  \alpha_{1}\left(  \lambda\right)  ,\alpha_{2}\left(  \lambda\right)
,\alpha_{3}\left(  \lambda\right)  ,\alpha_{4}\left(  \lambda\right)
,1\right)  =5$ almost everywhere. Thus, $\pi_{\lambda}|_{\Gamma}$ is almost everywhere irreducible.
\end{example}

\section{Proof of Theorem \ref{main2}}

Suppose that $m$ is odd. For any linear functional $\lambda\in\mathfrak{f}%
_{m,2}^{\ast},$ we consider the corresponding skew-symmetric matrix
\begin{equation}
\mathbf{M}\left(  \lambda\right)  =\left[
\begin{array}
[c]{cc}%
0_{\frac{m\left(  m-1\right)  }{2}\times\frac{m\left(  m-1\right)  }{2}} &
0_{\frac{m\left(  m-1\right)  }{2}\times m}\\
0_{m\times\frac{m\left(  m-1\right)  }{2}} & \left[  \lambda\left[
Z_{i},Z_{j}\right]  \right]  _{1\leq i,j,m}%
\end{array}
\right]  .\label{M}%
\end{equation}
Since $m$ is odd then $m=2k+1$ for some positive integer $k$ greater than or
equal to one. If $\left[  \lambda\left[  Z_{i},Z_{j}\right]  \right]  _{1\leq
i,j,m}^{T}$ denotes the transpose of $\left[  \lambda\left[  Z_{i}%
,Z_{j}\right]  \right]  _{1\leq i,j,m},$ it is easy to see that the null-space
of $\mathbf{M}\left(  \lambda\right)  $ is equal to
\[
\mathfrak{z}\oplus\text{ \textrm{nullspace} }\left(  \left[  \lambda\left[
Z_{i},Z_{j}\right]  \right]  _{1\leq i,j,m}\right)  =\mathfrak{z}\oplus\text{
\textrm{nullspace} }\left(  \left[  \lambda\left[  Z_{i},Z_{j}\right]
\right]  _{1\leq i,j,m}^{T}\right)
\]
since $\left[  \lambda\left[  Z_{i},Z_{j}\right]  \right]  _{1\leq i,j,m}$ is
a skew-symmetric matrix. Now, let 
\[
\alpha=\left[
\begin{array}
[c]{c}%
\alpha_{1}\\
\vdots\\
\alpha_{2k}%
\end{array}
\right]  ,\text{ and }\beta=\left[
\begin{array}
[c]{c}%
\beta_{1}\\
\vdots\\
\beta_{2k}%
\end{array}
\right]  \in%
%TCIMACRO{\U{211d} }%
%BeginExpansion
\mathbb{R}
%EndExpansion
^{2k}.
\]
Also, let $\left\langle \cdot,\cdot\right\rangle $ be the standard inner
product defined on $%
%TCIMACRO{\U{211d} }%
%BeginExpansion
\mathbb{R}
%EndExpansion
^{2k}.$

\begin{lemma}
\label{skew} If $M$ is a skew-symmetric matrix in $GL\left(  2k,%
%TCIMACRO{\U{211d} }%
%BeginExpansion
\mathbb{R}
%EndExpansion
\right)  $ and if $M\alpha=\beta$ then $\left\langle \alpha,\beta\right\rangle
=0$
\end{lemma}

\begin{proof}
Suppose that $M\alpha=\beta.$ Then
\begin{align*}
\left.  M\alpha=\beta\Rightarrow\right.  \left\langle \alpha,\beta
\right\rangle  &  =\left\langle \alpha,MM^{-1}\beta\right\rangle \\
&  =\left\langle M^{T}\alpha,M\beta\right\rangle \\
&  =\left\langle -M\alpha,M^{-1}\beta\right\rangle \\
&  =\left\langle -\beta,\alpha\right\rangle \\
&  =-\left\langle \alpha,\beta\right\rangle .
\end{align*}
We conclude that $\left\langle \alpha,\beta\right\rangle =0.$
\end{proof}

Let us define the following vector-valued functions on $\mathfrak{f}%
_{m,2}^{\ast}$:
\[
\lambda\mapsto\alpha\left(  \lambda\right)  =\left[
\begin{array}
[c]{c}%
\alpha_{1}\left(  \lambda\right) \\
\vdots\\
\alpha_{2k}\left(  \lambda\right)
\end{array}
\right]  \text{ and }\gamma\mapsto\gamma\left(  \lambda\right)  =\left[
\begin{array}
[c]{c}%
\alpha\left(  \lambda\right) \\
1
\end{array}
\right]  .
\]

\begin{lemma}
\label{null} For almost every linear functional $\lambda,$ there exists some
$\gamma\left(  \lambda\right)  \in%
%TCIMACRO{\U{211d} }%
%BeginExpansion
\mathbb{R}
%EndExpansion
^{2k+1}$ such that
\[
\mathrm{nullspace}\left(  \mathbf{M}\left(  \lambda\right)  \right)
=\mathfrak{z}\oplus%
%TCIMACRO{\U{211d} }%
%BeginExpansion
\mathbb{R}
%EndExpansion
\left(  \gamma^{T}\left(  \lambda\right)  Z\right)
\]
\ where $Z^{T}=\left[
\begin{array}
[c]{ccc}%
Z_{1}, & \cdots &, Z_{2k+1}%
\end{array}
\right]  $.
\end{lemma}

\begin{proof}
We consider the equation
\begin{equation}
\left[  \lambda\left[  Z_{i},Z_{j}\right]  \right]  _{1\leq i,j,m}^{T}\left(
\gamma\right)  =0\label{El}%
\end{equation}
where
\[
\gamma=\left[
\begin{array}
[c]{c}%
\alpha\\
1
\end{array}
\right]  \in%
%TCIMACRO{\U{211d} }%
%BeginExpansion
\mathbb{R}
%EndExpansion
^{2k+1}\text{ is unknown.}%
\]
Equation (\ref{El}) is equivalent to the following system of $m$ equations and
$m-1$ unknowns%
\begin{equation}
\left\{
\begin{array}
[c]{c}%
\alpha_{1}\lambda\left[  Z_{1},Z_{1}\right]  +\alpha_{2}\lambda\left[
Z_{2},Z_{1}\right]  +\cdots+\alpha_{2k}\lambda\left[  Z_{2k},Z_{1}\right]
+\lambda\left[  Z_{2k+1},Z_{1}\right]  =0\\
\alpha_{1}\lambda\left[  Z_{1},Z_{2}\right]  +\alpha_{2}\lambda\left[
Z_{2},Z_{2}\right]  +\cdots+\alpha_{2k}\lambda\left[  Z_{2k},Z_{2}\right]
+\lambda\left[  Z_{2k+1},Z_{2}\right]  =0\\
\vdots\\
\alpha_{1}\lambda\left[  Z_{1},Z_{2k}\right]  +\alpha_{2}\lambda\left[
Z_{2},Z_{2k}\right]  +\cdots+\alpha_{2k}\lambda\left[  Z_{2k},Z_{2k}\right]
+\lambda\left[  Z_{2k+1},Z_{2k}\right]  =0\\
\alpha_{1}\lambda\left[  Z_{1},Z_{2k+1}\right]  +\alpha_{2}\lambda\left[
Z_{2},Z_{2k+1}\right]  +\cdots+\alpha_{2k}\lambda\left[  Z_{2k},Z_{2k+1}%
\right]  +\lambda\left[  Z_{2k+1},Z_{2k+1}\right]  =0
\end{array}
\right.  .\label{eq1}%
\end{equation}
We define a matrix-valued function $\lambda\mapsto M\left(  \lambda\right)  $
on $\mathfrak{f}_{m,2}^{\ast}$ such that
\[
M\left(  \lambda\right)  =\left[  \lambda\left[  Z_{i},Z_{j}\right]  \right]
_{1\leq i,j,m}^{T}=\left[
\begin{array}
[c]{cccc}%
\lambda\left[  Z_{1},Z_{1}\right]   & \lambda\left[  Z_{2},Z_{1}\right]   &
\cdots & \lambda\left[  Z_{2k},Z_{1}\right]  \\
\lambda\left[  Z_{1},Z_{2}\right]   & \lambda\left[  Z_{2},Z_{2}\right]   &
\cdots & \lambda\left[  Z_{2k},Z_{2}\right]  \\
\vdots & \vdots & \ddots & \vdots\\
\lambda\left[  Z_{1},Z_{2k}\right]   & \lambda\left[  Z_{2},Z_{2k}\right]   &
\cdots & \lambda\left[  Z_{2k},Z_{2k}\right]
\end{array}
\right]  \text{ and }%
\]%
\[
\beta\left(  \lambda\right)  =\left[
\begin{array}
[c]{c}%
-\lambda\left[  Z_{2k+1},Z_{1}\right]  \\
-\lambda\left[  Z_{2k+1},Z_{2}\right]  \\
\vdots\\
-\lambda\left[  Z_{2k+1},Z_{2k}\right]
\end{array}
\right]  .
\]
Let $\Omega=\left\{  \lambda\in\mathfrak{f}_{m,2}^{\ast}:\det M\left(
\lambda\right)  \neq0\right\}  .$ $\Omega$ is a Zariski open subset of
$\mathfrak{f}_{m,2}^{\ast}$ since $\det M\left(  \lambda\right)  $ is a
non-trivial homogeneous polynomial defined over $\mathfrak{f}_{m,2}^{\ast}.$
Therefore $\Omega$ is a dense and open subset of $\mathfrak{f}_{m,2}^{\ast}.$
Now, let $\lambda\in\Omega.$ Although Equation (\ref{eq1}) is equivalent to
solving the system of equations
\[
\left\{
\begin{array}
[c]{c}%
M\left(  \lambda\right)  \alpha=\beta\left(  \lambda\right)  \\
\left\langle \alpha,\beta\left(  \lambda\right)  \right\rangle =0
\end{array}
\right.  ,
\]
according to Lemma \ref{skew}, if $M\left(  \lambda\right)  \alpha
=\beta\left(  \lambda\right)  $ then
\begin{equation}
\left\langle \alpha,\beta\left(  \lambda\right)  \right\rangle
=0.\label{ortho}%
\end{equation}
The point here is that, in order to find a complete solution to Equation
(\ref{eq1}), we only need to solve $M\left(  \lambda\right)  \alpha
=\beta\left(  \lambda\right)  $ for $\alpha.$ Let $\alpha$ be a solution to
the equation. Since $M\left(  \lambda\right)  $ in non-singular, then
$\alpha=\alpha\left(  \lambda\right)  =M\left(  \lambda\right)  ^{-1}%
\beta\left(  \lambda\right)  $ and it follows that $\mathrm{nullspace}\left(
\mathbf{M}\left(  \lambda\right)  \right)  =\mathfrak{z}\oplus%
%TCIMACRO{\U{211d} }%
%BeginExpansion
\mathbb{R}
%EndExpansion
\left(  \gamma^{T}\left(  \lambda\right)  Z\right)  .$ This completes the proof.
\end{proof}

Now, let 
\[
\Omega_{1}=%
%TCIMACRO{\dbigcap \limits_{j=1}^{2k}}%
%BeginExpansion
{\displaystyle\bigcap\limits_{j=1}^{2k}}
%EndExpansion
\left\{  \lambda\in\mathfrak{f}_{m,2}^{\ast}:\lambda\left[  Z_{2k+1}%
,Z_{j}\right]  \neq0\right\}  .
\]
We observe that since $\Omega_{1}$ is Zariski open in $\mathfrak{f}%
_{m,2}^{\ast},$ it is a dense subset of $\mathfrak{f}_{m,2}^{\ast}.$

\begin{lemma}
For every linear functional $\lambda\in\Omega\cap\Omega_{1},$%
\[
\mathfrak{f}_{m,2}\left(  \lambda\right)  =\mathfrak{z}\oplus\text{ }%
%TCIMACRO{\U{211d} }%
%BeginExpansion
\mathbb{R}
%EndExpansion
\left(  \alpha_{1}\left(  \lambda\right)  Z_{1}+\alpha_{2}\left(
\lambda\right)  Z_{2}+\cdots+\alpha_{2k}\left(  \lambda\right)  Z_{2k}%
+Z_{2k+1}\right)
\]
and each function $\lambda\mapsto\alpha_{j}\left(  \lambda\right)  $ is a
non-vanishing rational function for all $1\leq j\leq2k.$
\end{lemma}

\begin{proof}
First, we notice that $\lambda\mapsto\beta\left(  \lambda\right)  $ is a
non-zero vector-valued function on $\Omega\cap\Omega_{1}.$ Let $\mathfrak{B=}%
\left\{  b_{1},\cdots,b_{2k}\right\}  $ be the canonical column basis for $%
%TCIMACRO{\U{211d} }%
%BeginExpansion
\mathbb{R}
%EndExpansion
^{2k}.$ Then the coordinates of $\beta\left(  \lambda\right)  :\left\langle
\beta\left(  \lambda\right)  ,b_{j}\right\rangle $ are non-zero monomials in $%
%TCIMACRO{\U{211d} }%
%BeginExpansion
\mathbb{R}
%EndExpansion
\left[  \lambda\left(  Z_{\left(  2k+1\right)  1}\right)  ,\cdots
,\lambda\left(  Z_{\left(  2k+1\right)  2k}\right)  \right]  .$ Since
$\beta\left(  \lambda\right)  ,\alpha\left(  \lambda\right)  $ are orthogonal
vectors (see (\ref{ortho})), then there exists some rotation
matrix $\theta\left(  \lambda\right)  $ in the orthogonal Lie group
$\mathrm{O}(2k,\mathbb{R})$ such that
\begin{equation}
\alpha\left(  \lambda\right)  =\frac{\left\Vert M\left(  \lambda\right)
^{-1}\beta\left(  \lambda\right)  \right\Vert }{\left\Vert \beta\left(
\lambda\right)  \right\Vert }\theta\left(  \lambda\right)  \beta\left(
\lambda\right)  . \label{rot}%
\end{equation}
If $\theta\left(  \lambda\right)  ^{-1}b_{k}=b_{l}\in\mathfrak{B}$ then
\begin{align*}
\left\langle \alpha\left(  \lambda\right)  ,b_{k}\right\rangle  &
=\left\langle \frac{\left\Vert M\left(  \lambda\right)  ^{-1}\beta\left(
\lambda\right)  \right\Vert }{\left\Vert \beta\left(  \lambda\right)
\right\Vert }\theta\left(  \lambda\right)  \beta\left(  \lambda\right)
,b_{k}\right\rangle \\
&  =\frac{\left\Vert M\left(  \lambda\right)  ^{-1}\beta\left(  \lambda
\right)  \right\Vert }{\left\Vert \beta\left(  \lambda\right)  \right\Vert
}\left\langle \beta\left(  \lambda\right)  ,b_{l}\right\rangle \\
&  \neq0\text{ for all }\lambda\in\text{ }\Omega\cap\Omega_{1}.
\end{align*}
We conclude that for any $1\leq j\leq m,$ $\alpha_{j}\left(  \lambda\right)  $
is a non-zero rational function of $\lambda.$ Moreover, for any $\lambda
\in\Omega\cap\Omega_{1},$ $
\mathfrak{f}_{m,2}\left(  \lambda\right)  =\mathrm{nullspace}\left(
\mathbf{M}\left(  \lambda\right)  \right)$ which is equal to  $\mathfrak{z}\oplus\text{ }%
%TCIMACRO{\U{211d} }%
%BeginExpansion
\mathbb{R}
%EndExpansion
\left(  \alpha_{1}\left(  \lambda\right)  Z_{1}+\alpha_{2}\left(
\lambda\right)  Z_{2}+\cdots+\alpha_{2k}\left(  \lambda\right)  Z_{2k}%
+Z_{2k+1}\right).$
\end{proof}
\vskip 0.5cm \noindent

\begin{proof}[Proof of Theorem \ref{main2}]
Let $\lambda\in\Omega\cap\Omega_{1}$ (a Zariski open and dense subset) and
define $X\in\mathfrak{f}_{2k+1,2}\left(  \lambda\right)  $ such that
\[
X=X_{1}+\cdots+X_{\frac{m\left(  m-1\right)  }{2}}+\alpha
_{1}\left(  \lambda\right)  Z_{1}+\alpha_{2}\left(  \lambda\right)
Z_{2}+\cdots+\alpha_{2k}\left(  \lambda\right)  Z_{2k}+Z_{2k+1}.
\]
Clearly for every $\lambda\in\Omega\cap\Omega_{1},$ 
\[
\dim_{%
%TCIMACRO{\U{211a} }%
%BeginExpansion
\mathbb{Q}
%EndExpansion
}\left(  \alpha_{1}\left(  \lambda\right)  ,\alpha_{2}\left(  \lambda\right)
,\cdots,\alpha_{2k}\left(  \lambda\right)  ,1\right)  =2k+1.
\]
So, the vector space $\mathfrak{f}_{m,2}\left(  \lambda\right)  $ is not
contained in a proper rational ideal of $\mathfrak{f}_{m,2}$ for all
$\lambda\in\Omega\cap\Omega_{1}.$ Appealing to Theorem \ref{Bekka}, the
corresponding irreducible representation $\pi_{\lambda}$ restricted to
$\Gamma$ is an \textbf{irreducible representation} of $\Gamma$ for all
$\lambda\in\Omega\cap\Omega_{1}.$
\end{proof}

\begin{remark}
If $\lambda\in\Omega\cap\Lambda$ such that
\[
\Lambda=%
%TCIMACRO{\dbigcap \limits_{j=1}^{2k}}%
%BeginExpansion
{\displaystyle\bigcap\limits_{j=1}^{2k}}
%EndExpansion
\left\{  \lambda\in\mathfrak{f}_{m,2}^{\ast}:\left[  Z_{2k+1},Z_{j}\right]
\in\ker\lambda\right\}
\]
then $\lambda\mapsto\beta\left(  \lambda\right)  $ is the zero function
defined on $\Omega\cap\Lambda$ and clearly,
\[
\mathfrak{f}_{m,2}\left(  \lambda\right)  =\mathrm{nullspace}\left(
\mathbf{M}\left(  \lambda\right)  \right)  =\mathfrak{z}\oplus\text{ }%
%TCIMACRO{\U{211d} }%
%BeginExpansion
\mathbb{R}
%EndExpansion
\left(  Z_{2k+1}\right)
\]
which is contained in a rational ideal of $\mathfrak{f}_{m,2}.$ Therefore,
each unitary irreducible representation $\pi_{\lambda}$ of $\exp\left(
\mathfrak{f}_{m,2}\right)  $ restricted to $\Gamma$ is a reducible
representation of $\Gamma$ for all $\lambda\in\Omega\cap\Lambda$ (meager
subset of $\mathfrak{f}_{m,2}^{\ast}$).
\end{remark}

\section{Proof of Theorem \ref{main1}}

Suppose that $m$ is even. This case is much easier than the odd case. The proof
will be very short. Since $m$ is even, then $m=2k$ for some positive integer
$k$ greater than or equal to one. Thus the null-space of $\mathbf{M}\left(
\lambda\right)  $ is equal to
\[
\mathfrak{z}\oplus\text{ }\mathrm{nullspace}\text{ }\left(  \left[
\lambda\left[  Z_{i},Z_{j}\right]  \right]  _{1\leq i,j,m}^{T}\right)  .
\]
Let $\lambda$ be an element of $\mathfrak{f}_{m,2}^{\ast}.$ We consider the
matrix equation
\[
\overset{%
\begin{array}
[c]{c}%
M\left(  \lambda\right)  =\left[  \lambda\left[  Z_{i},Z_{j}\right]  \right]
_{1\leq i,j,m}^{T}\\
\shortparallel
\end{array}
}{\overbrace{\left[
\begin{array}
[c]{cccc}%
\lambda\left[  Z_{1},Z_{1}\right]   & \lambda\left[  Z_{2},Z_{1}\right]   &
\cdots & \lambda\left[  Z_{2k},Z_{1}\right]  \\
\lambda\left[  Z_{1},Z_{2}\right]   & \lambda\left[  Z_{2},Z_{2}\right]   &
\cdots & \lambda\left[  Z_{2k},Z_{2}\right]  \\
\vdots & \vdots & \ddots & \vdots\\
\lambda\left[  Z_{1},Z_{2k}\right]   & \lambda\left[  Z_{2},Z_{2k}\right]   &
\cdots & \lambda\left[  Z_{2k},Z_{2k}\right]
\end{array}
\right]  }}\overset{%
\begin{array}
[c]{c}%
\alpha\in%
%TCIMACRO{\U{211d} }%
%BeginExpansion
\mathbb{R}
%EndExpansion
^{2k}\\
\shortparallel
\end{array}
}{\overbrace{\left[
\begin{array}
[c]{c}%
\alpha_{1}\\
\alpha_{2}\\
\vdots\\
\alpha_{2k}%
\end{array}
\right]  }}=\left[
\begin{array}
[c]{c}%
0\\
0\\
\vdots\\
0
\end{array}
\right]  
\]
which we want to solve for $\alpha.$ Put $\Omega=\left\{  \lambda\in\mathfrak{f}_{m,2}^{\ast}:\det M\left(
\lambda\right)  \neq0\right\}.$ For every $\lambda\in M\left(
\lambda\right)  ,$ we observe that $M\left(  \lambda\right)  $ is a
skew-symmetric matrix of even rank. Therefore $\Omega$ is a Zariski open
subset of $\mathfrak{f}_{n,2}^{\ast}$ and is dense and open in $\mathfrak{f}%
_{n,2}^{\ast}.$ Moreover, solving the given matrix equation above, we obtain that $\alpha=0.$ As a result, for any $\lambda\in\Omega,\mathfrak{f}%
_{m,2}\left(  \lambda\right)  =\mathfrak{z}.$ So, in summary, if the dimension
of the Lie algebra $\mathfrak{f}_{m,2}$ is even then $\mathfrak{f}_{m,2}\left(  \lambda
\right)  $ is contained inside a rational ideal of $\mathfrak{f}_{m,2}$ almost
everywhere. As a result, almost every irreducible representation $\pi
_{\lambda}$ of $\mathfrak{f}_{m,2}$ restricted to $\Gamma$ is a reducible
representation of $\Gamma.$

\end{document}